\documentclass[12pt]{article}
\usepackage{amsmath, amsthm, amscd, amsfonts, amssymb, graphicx}
\usepackage{titles}
\usepackage{algorithm}
\usepackage{algpseudocode}
\usepackage{subfigure} 
 \usepackage{caption}
 \usepackage{graphics,graphicx}

\usepackage{epsfig}
\usepackage{color}
\usepackage{fancyhdr}
\usepackage{titlesec}
\usepackage{tikz}
\usepackage{afterpage}
\usetikzlibrary{arrows}
\usetikzlibrary{decorations.markings}
\tikzstyle{block}=[draw opacity=0.7,line width=1.4cm]
\tikzset{
  big black arrow/.style={
    decoration={markings,mark=at position 1 with {\arrow[scale=2.5,black]{>}}},
    postaction={decorate},
    shorten >=0.4pt},
    line/.style={draw, ->}}

\setbox0=\hbox{$+$}
\newdimen\plusheight
\plusheight=\ht0
\def\+{\;\lower\plusheight\hbox{$+$}\;}

\setbox0=\hbox{$-$}
\newdimen\minusheight
\minusheight=\ht0
\def\-{\;\lower\minusheight\hbox{$-$}\;}

\setbox0=\hbox{$\cdots$}
\newdimen\cdotsheight
\cdotsheight=\plusheight
\def\cds{\lower\cdotsheight\hbox{$\cdots$}}

\renewcommand{\(}{\left\(}
\renewcommand{\)}{\right\)}

\renewcommand{\pmod}[1]{\,(\textup{mod}\,#1)}
\numberwithin{equation}{section}
\theoremstyle{plain}
\newtheorem{theorem}{Theorem}[section]
\newtheorem{lemma}[theorem]{Lemma}

\newtheorem{remarks}[theorem]{Remark}

\usepackage{pdfpages}

\begin{document}
\begin{center}{\bf Partition-Theoretic Results and Recurrence Relations for the Coefficients of Some Mock Theta Functions
	}\end{center}
	\begin{center}	
	\footnotesize{\bf Sabi Biswas and Nipen Saikia$^{\ast}$}\\
					Department of Mathematics, Rajiv Gandhi
			University,\\ Rono Hills, Doimukh-791112, Arunachal Pradesh, India.\\
			E. Mail(s): sabi.biswas@rgu.ac.in; nipennak@yahoo.com\\
			$^\ast$\textit{Corresponding author}.\end{center}\vskip2mm
			
			\noindent {\bf Abstract:} In this paper, we give  partition-theoretic results for the coefficients of some mock theta functions and  prove their congruence properties. Some recurrence relations connecting the coefficients of the mock theta functions with certain restricted partition functions are also established.
			
\section{Introduction} 
In 1920, Ramanujan introduced $17$ functions in his last letter to G. H. Hardy~\cite[p. 534]{R}, which he called as mock theta functions. Initially, Ramanujan divided his list of mock theta functions into odd order as three, five and seven.
 After Ramanujan, many new mock theta-functions are defined and studied by different mathematicians. An account of these can be found in the papers by  Andrews~\cite{A1}, Andrews and Hickerson \cite{A2}, Gordon and  McIntosh~\cite{G1}, K. Hikami~\cite{HK} and references therein.  In this paper, we are interested in following mock theta functions of order two, six and eight:
\begin{align}
\label{u}
&\mu(q)=\sum_{n=0}^{\infty}\dfrac{(-1)^nq^{n^2}(q;q^2)_n}{(-q^2;q^2)_{n}^2},\\
\label{s}
&\sigma(q)=\sum_{n=0}^{\infty}\dfrac{q^{(n+1)(n+2)/2}(-q;q)_n}{(q;q^2)_{n+1}},\\
\label{b}
&\beta(q)=\sum_{n=0}^{\infty}\dfrac{q^{3n^2+3n+1}}{(q;q^3)_{n+1}(q^2;q^3)_{n+1}},\\
\label{l}
&\lambda(q)=\sum_{n=0}^{\infty}\dfrac{(-1)^nq^n(q;q^2)_n}{(-q;q)_n},\\
\label{v}
& v(q)=\sum_{n=0}^{\infty}\dfrac{q^{(n+1)^2}(-q;q^2)_n}{(q;q^2)_{n+1}},
\end{align}where $$(b;q)_n:=\prod_{k=0}^{n-1}(1-bq^k)\quad \text{and}\quad (b;q)_\infty:=\prod_{k=0}^{\infty}(1-bq^k).$$ For brevity, we will  write,  for any positive  integer $k$, $$\ell _{k} := (q^k;q^k)_{\infty}$$ and
$$ 
(b_{1},b_{2},b_{3},...,b_{k};q)_{\infty}:=(b_{1};q)_{\infty}(b_{2};q)_{\infty}(b_{3};q)_{\infty}....(b_{k};q)_{\infty}.$$
The function defined in \eqref{v} is a eighth order mock theta function defined in \cite[pp. 322-323]{G1}. Agarwal and Sood~\cite{AS} gave combinatorial interpretation of $v(q)$ using  split $(n + t)$-color partitions. Rana and Sareen~\cite{D} extended their results using signed partitions.

The function defined in \eqref{u} is the second order mock theta function which appeared in Ramanujan's lost notebook~\cite{R} (see also~\cite{A1}). Combinatorial interpretation of $\mu(q)$ was given by Kaur and Rana in~\cite{RK1}.

The functions defined in \eqref{s}-\eqref{l} are the sixth order mock theta functions. In 2018, Zhang~\cite{ZS} established some congruences modulo $3$, $5$ and $7$ for the coefficients of the mock theta fucntion $\beta(q)$. Kaur and Rana~\cite{RK} proved some particular and infinite families of congruences for the coefficients of the mock theta fucntion $\lambda(q)$.

In this paper, we prove partition-theoretic results for the coeffcients of the mock theta functions defined in \eqref{s}-\eqref{v}, and establish their congruence properties. We also prove some recurrence relations connecting the coefficients of the mock theta functions and certain restricted partition functions. The results on mock theta functions $v(q)$,  $\sigma(q)$,  $\beta(q)$ and $\lambda(q)$ are established in Sections 3, 4, 5,  and 6. respectively. Section 2 is devoted to record some preliminary results which will be used in the subsequent sections.

We end this introduction by defining  partition of a positive integer,  colour partition  of a positive integer, and  their generating functions. A partition of an positive integer $n$ can be defined as finite sequence of positive integers $(\delta_1, \delta_2, \cdots, \delta_k)$ such that $\sum_{j=1}^k\delta_j=n;\quad \delta_j\ge \delta_{j+1},$ where $\delta_j$ are called parts or summands of the partition. The number of partitions of $n$ is usually denoted by $p(n)$. The generating function for the partition function $p(n)$ is given by Euler~\cite{LE} as \begin{equation}\label{pn}\sum_{n=0}^{\infty}p(n)q^{n} = \frac{1}{(q;q)_{\infty}}=\frac{1}{\ell_1}, \qquad p(0)=1.\end{equation}
Euler~\cite{AA} provided the following recurrence relation for finding the values of the partition function $p(n)$:
\begin{align*}
& p(n)-p(n-1)-p(n-2)+p(n-5)+p(n-7)-p(n-12)-p(n-15)+.....\\
&+(-1)^kp\Big(n-k(3k-1)/2\Big)+(-1)^kp\Big(n-k(3k+1)/2\Big)+...= \left\{\begin{array}{cc}
1, \quad
if~ n=0, \\
\hspace{.3cm}0, \quad \mbox{otherwise}.
\end{array}
\right.
\end{align*}where the numbers of the form $k(3k\pm1)/2$ are called pentagonal numbers. For more recurrence relations of different partition functions one can see \cite{EJ, N, O, S}. Ramanujan ~\cite{RS} offered following congruences for $p(n)$:
$$p(5n+4)\equiv 0\pmod 5, \quad   p(7n+5)\equiv 0\pmod 7\quad \mbox{and}\quad    p(11n+6)\equiv 0\pmod{11}.$$
In a letter to Hardy, Ramanujan \cite{BR} introduced general partition function $p_r(n)$ as
\begin{equation}\label{e7}
\sum_{n=0}^{\infty}p_r(n)q^n=\dfrac{1}{(q;q)_{\infty}^r}=\frac{1}{\ell_1^r},
\end{equation}where $r$ is any non-zero integer. $p_1(n)$ is the partition function $p(n)$ defined in \eqref{pn}. For $r>1$,  $p_{r}(n)$ is the  colour partition function of an integer $n\ge 1$ in which each part in the partitions of $n$ is assumed to have $r$ different colours and all of them are considered as distinct. 
For $r=-1$, $p_r(n)$ is the  Euler's pentagonal number theorem and can be combinatorially interpretated in terms of partition of integers as
\begin{equation}\label{mr}p_{-1}(n)=p(n, e)-p(n, o)=\begin{cases}(-1)^m  & \text{if}~ n=m(3m\pm1)/2 \\
0 & \text{otherwise}	
\end{cases},\end{equation} where $p(n, e)$ (resp. $p(n, o)$) is the number of partitions of $n$ with even (resp. odd) number of distinct parts. For $r<-1$, it is easily seen that \begin{equation}\label{mr1}p_{r}(n)=p_r(n, e)-p_r(n, o),\end{equation}where $p_{r}(n, e)$ (resp. $p_{r}(n, o)$) is the number of partitions of $n$ with even (resp. odd) number of distinct parts and each part has $r$ colours. It  is useful to note here that, for positive integers $r$, $s$ and $t$, $$\dfrac{1}{(q^t;q^s)^r}$$ denotes the generating function of the number of partitions of a positive integer such that parts are $\equiv t\pmod{s}$  and each part has  $r$ distinct colours.

To prove  recurrence relations,  we will also use the  restricted partition functions $\bar{p}_{r}(n)$, $p_{rd}(n)$ and $A_{4}(n)$, where $\bar{p}_{r}(n)$ denotes the number of overpartitions of $n$ with $r$ copies, $p_{rd}(n)$ denotes the number of partition of $n$ into distinct parts with $r$ copies and $A_{4}(n)$ denotes the number of $4$-regular partitions of $n$ respectively. The corresponding generating functions are given by 
\begin{align}\label{p1}
&\sum_{n=0}^{\infty}\bar{p}_{r}(n)q^n={\left(\dfrac{\ell_2}{\ell_1^2}\right)}^r,\\
\label{p2}
&\sum_{n=0}^{\infty}p_{rd}(n)q^n={\left(\dfrac{\ell_2}{\ell_1}\right)}^r,\\
\label{p3}
&\sum_{n=0}^{\infty}A_4(n)q^n=\dfrac{\ell_4}{\ell_1}.
\end{align}

\section{Preliminaries}
Ramanujan defined general theta-function $f(c,d)$ \cite[p. 34, (18.1)]{BBC} as
\begin{equation}\label{eq3}
f(c,d) = \sum_{m={-}\infty}^{\infty}c^{m(m+1)/2}d^{m(m-1)/2},  \quad                |cd|<1.
\end{equation}Three useful special cases \cite[p. 35, Entry 18]{BBC} of $f(c,d)$ are given by 
\begin{equation}\label{eq4}
\phi(q) := f(q,q) = \sum_{m={-}\infty}^{\infty}q^{{m}^2}= \frac{\ell_{2}^{5} }{{\ell_{1}^2}{\ell_{4}^2}},
\end{equation} 
\begin{equation}\label{eq5}
\psi(q) := f(q,q^3) = \sum_{m=0}^{\infty}q^{m(m+1)/2} = \frac{\ell_{2}^2}{\ell_{1}}
\end{equation}and
\begin{equation}\label{eq6}
f(-q) := f(-q,-q^2) = \sum_{m=-\infty}^{\infty}(-1)^m{q}^{m(3m-1)/2} = \ell_{1}.
\end{equation} The product representations in right hand side of \eqref{eq4}-\eqref{eq6} are the consequences of the Jacobi's triple product identity given by
\begin{equation}\label{jaco}f(c,d)= (-c;cd)_{\infty}(-d;cd)_{\infty}(cd;cd)_{\infty}.\end{equation}By using elementary $q$-operations, it is easily seen that
\begin{equation}\label{J1}
\phi(-q)=\frac{(q;q)_\infty}{(-q;q)_\infty} = \frac{(q;q)_{\infty}^2}{(q ^2;q^2)_{\infty}} = \frac{\ell_{1}^2}{\ell_{2}}.
\end{equation}

In some of the proofs, we will also  use the  Jacobi's identity~\cite[p. 39, Entry 24]{BBC} given by 
\begin{equation}\label{j}
 \ell_1^3=\sum_{k=0}^{\infty}(-1)^k(2k+1)q^{k(k+1)/2}.\end{equation}
 
 \begin{lemma} [{\cite[Theorem 2.1]{CG}}] If p is an odd prime, then
 \begin{equation}\label{eq7}
 \psi(q) = \sum_{m=0}^{(p-3)/2} q^{(m^{2}+m)/2} f\left(q^{\left(p^{2}+(2m+1)p\right)/2}, q^{\left(p^{2}-(2m+1)p\right)/2}\right) + q^{(p^{2}-1)/8}\psi(q^{p^2}).
 \end{equation}
 Furthermore, $\dfrac{m^{2}+m}{2} \not\equiv \dfrac{p^{2}-1}{8} \pmod{p}\quad for\quad 0\leq m \leq \dfrac{p-3}{2}.$
 \end{lemma}
 \begin{lemma}[{\cite[Theorem 2.2]{CG}}] If $p\geq5$ is a prime, then
 \begin{multline}\label{eq8}
 \ell_1=\sum_{\substack{t={-(p-1)/2} \\ t \ne {(\pm p-1)/6}}}^{(p-1)/2} (-1)^{t} q^{(3t^{2}+t)/2} f\left(-q^{\left(3p^{2}+(6t+1)p\right)/2}, -q^{\left(3p^{2}-(6t+1)p\right)/2}\right)\\ + (-1)^{(\pm p-1)/6} q^{(p^{2}-1)/24} f_{p^2},
 \end{multline} where
 \begin{equation*}
 \dfrac{\pm p-1}{6}
 = \left\{
 \begin{array}{cc}
 \dfrac{(p-1)}{6} \quad
  if~ p \equiv 1\pmod{6}, \\
 \dfrac{(-p-1)}{6} \quad
  if~ p \equiv -1\pmod{6}.
 \end{array}
 \right.
 \end{equation*} Furthermore, if $\dfrac{-p-1}{2} \leq t \leq \dfrac{p-1}{2} \quad and \quad t \ne \dfrac{\pm p-1}{2}, \quad then  \quad \dfrac{3t^{2}+t}{2} \not\equiv \dfrac{p^{2}-1}{24} \pmod{p}.$
 \end{lemma}
 
  \begin{lemma}[{\cite[Lemma 2.3]{AB}}] If p is an odd prime, then
 \begin{multline}\label{eq9}
 \ell_{1}^3=\sum_{\substack{k=0 \\ k\ne (p-1)/2}}^{(p-1)} (-1)^k q^{k(k+1)/2} \sum_{n=0}^{\infty} (-1)^{n} (2pn+2k+1) q^{pn(pn+2k+1)/2} \\ + p(-1)^{(p-1)/2} q^{(p^{2}-1)/8} f_{p^2}^3.
 \end{multline} Furthermore, if $k \ne \dfrac{p-1}{2},\quad 0\leq k \leq p-1, \quad then\quad \dfrac{k^{2}+k}{2} \not\equiv \dfrac{p^{2}-1}{8} \pmod{p}.$
 \end{lemma}
 \begin{lemma} We have
 \begin{align}
 \label{e2} 
 &\dfrac{\ell_2}{\ell_1^2}=\dfrac{\ell_6^4 \ell_9^6}{\ell_3^8 \ell_{18}^3}+2q\dfrac{\ell_6^3 \ell_9^3}{\ell_3^7}+4q^2\dfrac{\ell_6^2 \ell_{18}^3}{\ell_3^6},\\
\label{e4}
&\dfrac{1}{\ell_1\ell_2}=\dfrac{\ell_9^9}{\ell_3^6\ell_6^2\ell_{18}^3}+q\dfrac{\ell_9^6}{\ell_3^5\ell_6^3}+3q^2\dfrac{\ell_9^3\ell_{18}^3}{\ell_3^4\ell_6^4}-2q^3\dfrac{\ell_{18}^6}{\ell_3^3\ell_6^5}+4q^4\dfrac{\ell_{18}^9}{\ell_3^2\ell_6^6\ell_9^3},\\
\label{c1}
&\dfrac{\ell_4}{\ell_1}=\dfrac{\ell_{12}\ell_{18}^4}{\ell_3^3\ell_{36}^2}+q\dfrac{\ell_6^2\ell_9^3\ell_{36}}{\ell_3^4\ell_{18}^2}+2q^2\dfrac{\ell_6\ell_{18}\ell_{36}}{\ell_3^3}.
 \end{align}
 \end{lemma}
 The proof of \eqref{e2} can be seen from \cite{HD1}. The identity \eqref{e4} was obtained by Chan \cite{CH}. The identity \eqref{c1} follows from Equations (33.2.1) and (33.2.5) in \cite{HD}.

In addition to above identities, we need the following congruences which are easy consequences of the binomial theorem: 
 For any prime $p$ and positive integers $n$ and $t$, we have
 	\begin{align}\label{lm1}
 	& \ell_{n}^{tp} \equiv \ell_{np}^t\pmod{p},\\
 	\label{lm2}
 	& \ell_1^{2^t}\equiv\ell_2^{2^{t-1}}\pmod{2^t}.
 	\end{align}
 	In order to state our congruences, we will also use Legendre's symbol which is defined as follows:
 	
 	 Let $p$ be any odd prime and $\omega$ be any integer relatively prime to $p$, then the Legendre's symbol $\left(\dfrac{\omega}{p}\right)$ is defined by  
 	\begin{equation*}
 	\left(\dfrac{\omega}{p}\right)
 	= \left\{
 	\begin{array}{cc}
 	\hspace{-.5cm}1,~\text{if $\omega$ is quadratic residue of $p$}, \\
 	-1,~\text{if $\omega$ is quadratic non-residue of $p$} .
 	\end{array}
 	\right.
 	\end{equation*}
 	We will also use the notation, for any real number $x$, $$\lfloor x\rfloor=k,\quad \mbox{where}\quad k\leq x<k+1\quad \mbox{and}\quad k\quad \mbox{is an integer}.$$

\section{Results on $v(q)$}
Throughout the section, we set $\sum_{n=0}^{\infty}P_v(n)q^n=v(q)$, where $v(q)$ is defined in \eqref{v}.
\begin{theorem}\label{thm1}We have
$$\sum_{n=0}^{\infty}P_v(2n+1)q^n=\dfrac{\ell_4^3}{\ell_1\ell_2}.$$
\end{theorem}
\begin{proof}From Ramanujan's lost notebook \cite[p. 280, Entry 12.5.1]{BB} and \cite[p. 288]{RM}, we note that 
	\begin{equation}\label{e5}
	\mu(-q^2)+4v(q)=\dfrac{(q^4;q^4)_\infty(-q^2;q^4)_{\infty}^3}{(q^2;q^4)_{\infty}^2(-q^4;q^4)_{\infty}^2}+4q\dfrac{(q^8;q^8)_{\infty}(-q^4;q^4)_\infty}{(q^4;q^8)_\infty(q^2;q^4)_\infty},  \end{equation}
	where $p_\mu(q)$ is defined as in \eqref{u}. 
Simplifying \eqref{e5}, we obtain
\begin{equation}\label{a1}
4v(q)=	-\mu(-q^2)+\dfrac{(q^4;q^4)_\infty(-q^2;q^4)_{\infty}^3}{(q^2;q^4)_{\infty}^2(-q^4;q^4)_{\infty}^2}+4q\dfrac{(q^8;q^8)_{\infty}(-q^4;q^4)_\infty}{(q^4;q^8)_\infty(q^2;q^4)_\infty}
\end{equation}Extracting the terms involving $q^{2n+1}$ from \eqref{a1} and dividing by $q$ and then replacing $q^2$ by $q$, we obtain
\begin{equation}\label{a2}
\sum_{n=0}^{\infty}P_v(2n+1)q^n=\dfrac{(q^4;q^4)_{\infty}(-q^2;q^2)_\infty}{(q^2;q^4)_\infty(q;q^2)_\infty}
\end{equation}
Now the desired results easily follows from \eqref{a2}.
\end{proof}

\begin{theorem}\label{thmv}
We have $$P_v(2n+1)={P_v}_{(d,e)}(n)-{P_v}_{(d,0)}(n),$$ where ${P_v}_{(d,e)}(n)\left(resp.\quad {P_v}_{(d,0)}(n)\right)$ denotes the number of partitions of $n$ such that
\begin{itemize}
\item[$(i)$]parts congruent to $1,3\pmod{4}$ have one colour,
\item[$(ii)$]parts congruent to $2\pmod{4}$ have two colour, 
\item[$(iii)$]parts congruent to $0\pmod{4}$ are distinct and have one colour,
\item[$(iv)$]number of parts congruent to $0\pmod{4}$ being even (resp. odd).
\end{itemize}
\end{theorem}
\begin{proof}From Theorem \ref{thm1}, we have
\begin{equation}\label{3}
\sum_{n=0}^{\infty}P_v(2n+1)q^n=\dfrac{\ell_4^3}{\ell_1\ell_2}.
\end{equation}Changing the bases to $q^4$ in \eqref{3}, we obtain
\begin{equation}\label{v3}
\sum_{n=0}^{\infty}P_v(2n+1)q^n=\dfrac{(q^4;q^4)_{\infty}^3}{(q,q^2,q^3,q^4;q^4)_\infty(q^2,q^4;q^4)_\infty},
\end{equation}which implies
\begin{equation}\label{v4}
\sum_{n=0}^{\infty}P_v(2n+1)q^n=\dfrac{(q^4;q^4)_{\infty}}{(q,q^3;q^4)_\infty(q^2;q^4)_{\infty}^2}.
\end{equation}Now, the result immediately follows from \eqref{v4}.
\end{proof}
\begin{theorem}\label{thm2}We have
$$\hspace{-11cm}(i)\quad P_v(6n+5)\equiv0\pmod{3}.$$

\hspace{-.5cm}$(ii)$~Let $p\ge3$ be any prime such that $\left(\dfrac{-2}{p}\right)=-1$. Then for any integer $\alpha\ge0$, we have
\begin{equation}\label{7}\hspace{-4.5cm}
\sum_{n=0}^{\infty}P_v\left(2\cdot p^{2\alpha}n+\dfrac{3\cdot p^{2\alpha}+1}{4}\right)q^n\equiv\psi(q)\psi(q^2)\pmod{2},
\end{equation}
\begin{equation}\label{e8}\hspace{1.2cm}
P_v\left(2\cdot p^{2\alpha+2}n+2\cdot p^{2\alpha+1}j+\dfrac{3\cdot p^{2\alpha+2}+1}{4}\right)\equiv0\pmod{2},\quad 1\leq j\leq(p-1).
\end{equation}

\hspace{-.6cm}$(iii)$~Let $p\ge5$ be any prime such that $\left(\dfrac{-18}{p}\right)=-1$. Then for integer $\alpha\ge0$, we have
\begin{equation}\label{e9}\hspace{-5cm}
\sum_{n=0}^{\infty}P_v\left(6\cdot p^{2\alpha}n+\dfrac{19\cdot p^{2\alpha}+1}{4}\right)q^n\equiv3\ell_1\ell_6^3\pmod{6}.\end{equation}
\begin{equation}\label{e10}\hspace{1.2cm}
P_v\left(6\cdot p^{2\alpha+2}n+6\cdot p^{2\alpha+1}j+\dfrac{19\cdot p^{2\alpha+2}+1}{4}\right)\equiv0\pmod{6},\quad 1\leq j\leq(p-1).
\end{equation} 
\end{theorem}
\begin{proof} 
(i)~Using \eqref{c1} in \eqref{3}, we obtain
\begin{equation}\label{1}
\sum_{n=0}^{\infty}P_v(2n+1)q^n=\dfrac{\ell_{12}^2\ell_{18}^6}{\ell_3^3\ell_6\ell_{36}^3}+q\dfrac{\ell_{12}\ell_6\ell_9^3}{\ell_3^4}+3q^2\dfrac{\ell_{12}\ell_{18}^3}{\ell_3^3}+q^3\dfrac{\ell_6^2\ell_9^3\ell_{36}^3}{\ell_3^4\ell_{18}^3}+2q^4\dfrac{\ell_6\ell_{36}^3}{\ell_3^3}. 
\end{equation}Extracting the terms involving $q^{3n+2}$, dividing by $q^2$ and replacing $q^3$ by $q$ from \eqref{1}, we obtain
\begin{equation}\label{2}
\sum_{n=0}^{\infty}P_v(6n+5)q^n=3\dfrac{\ell_4\ell_6^3}{\ell_1^3}.
\end{equation} Now (i) follows immediately from \eqref{2}.

(ii)~Using \eqref{lm1} in \eqref{3}, we obtain
\begin{equation}\label{e12}
\sum_{n=0}^{\infty}P_v(2n+1)q^n\equiv\psi(q)\psi(q^2)\pmod{2}.
\end{equation}which is the $\alpha=0$ case of \eqref{7}. Assume that \eqref{7} is true for some $\alpha\geq0$. Now, substituting  \eqref{eq7} in  \eqref{7}, we obtain
\begin{multline}\label{e13}
\sum_{n=0}^{\infty}P_v\left(2\cdot p^{2\alpha}n+\dfrac{3\cdot p^{2\alpha}+1}{4}\right)q^n\\ \hspace{-.2cm}\equiv \Big[\sum_{m=0}^{(p-3)/2} q^{(m^{2}+m)/2} f\left(q^{\left(p^{2}+(2m+1)p\right)/2}, q^{\left(p^{2}-(2m+1)p\right)/2}\right) + q^{(p^{2}-1)/8}\psi(q^{p^2})\Big]\\ \times \Big[\sum_{k=0}^{(p-3)/2} q^{(k^{2}+k)} f\left(q^{\left(p^{2}+(2k+1)p\right)}, q^{\left(p^{2}-(2k+1)p\right)}\right) + q^{(p^{2}-1)/4}\psi(q^{2p^2})\Big]\pmod{2}.
\end{multline}Consider the congruence
\begin{equation}\label{e14}
\left(\dfrac{m^2+m}{2}\right)+\left({k^2+k}\right)\equiv 3\left(\dfrac{p^2-1}{8}\right)\pmod{p},
\end{equation}which is equivalent to
\begin{equation}\label{e15}
(2m+1)^2+2(2k+1)^2\equiv0\pmod{p}.
\end{equation}For $\left(\dfrac{-2}{p}\right)=-1$, the congruence \eqref{e15} has only solution $m=k=\dfrac{(p-1)}{2}$. Therefore, extracting the terms involving $q^{pn+3(p^{2}-1)/8}$, divding by $q^{3(p^{2}-1)/8}$ and replacing $q^p$ by $q$ from \eqref{e13}, we obtain
\begin{equation}\label{e16}
\sum_{n=0}^{\infty}P_v\left(2\cdot p^{2\alpha+1}n+\dfrac{3\cdot p^{2\alpha+2}+1}{4}\right)q^n\equiv\psi(q^{p})\psi(q^{2p})\pmod{2}.
\end{equation}Extracting the terms involving $q^{pn}$ and replacing $q^p$ by $q$ from \eqref{e16}, we obtain
\begin{equation*}
\sum_{n=0}^{\infty}P_v\left(2\cdot p^{2\alpha+2}n+\dfrac{3\cdot p^{2\alpha+2}+1}{4}\right)q^n\equiv\psi(q)\psi(q^2)\pmod{2},
\end{equation*}which is the $\alpha+1$ case of \eqref{7}. Hence, by the method of induction we complete the proof of \eqref{7}. Again, extracting the terms involving $q^{pn+j}$, $1\leq j\leq(p-1)$ from \eqref{e16}, we arrive at \eqref{e8}. 

(iii)~Using \eqref{lm2} in \eqref{2}, we obtain
\begin{equation}\label{e18}
\sum_{n=0}^{\infty}P_v(6n+5)q^n\equiv3\ell_1\ell_6^3\pmod{6}.
\end{equation}which is the $\alpha=0$ case of \eqref{e9}. Assume that \eqref{e9} is true for some $\alpha\geq0$. Now, substituting \eqref{eq8} and \eqref{eq9} in \eqref{e9}, we obtain
\begin{multline}\label{22}
\sum_{n=0}^{\infty}P_v\left(6\cdot p^{2\alpha}n+\dfrac{19\cdot p^{2\alpha}+1}{4}\right)q^n\\ \hspace{.2cm}\equiv  3\Big[\sum_{\substack{t={-(p-1)/2} \\ t \ne {(\pm p-1)/6}}}^{(p-1)/2} (-1)^{t} q^{(3t^{2}+t)/2} f\left(-q^{3p^{2}+(6t+1)p/2}, -q^{3p^{2}-(6t+1)p/2}\right)+(-1)^{(\pm p-1)/6} q^{(p^{2}-1)/24} f_{p^2}\Big]\\\hspace{-3.5cm} \times \Big[\sum_{\substack{k=0 \\ k\ne (p-1)/2}}^{(p-1)} (-1)^k q^{3k(k+1)} \sum_{n=0}^{\infty} (-1)^{n} (2pn+2k+1) q^{6pn(pn+2k+1)/2} \\ + p(-1)^{(p-1)/2} q^{6(p^{2}-1)/8} f_{6p^2}^3\Big]\pmod{6}.
\end{multline} 
Consider the congruence
\begin{equation}\label{23}
\left(\dfrac{3t^2+t}{2}\right)+(3k^2+3k)\equiv 19\left(\dfrac{p^2-1}{24}\right)\pmod{p},
\end{equation}which is equivalent to \begin{equation}\label{24}
(6t+1)^2+18(2k+1)^2\equiv0\pmod{p}.
\end{equation}For $\left(\dfrac{-18}{p}\right)=-1$, the congruence \eqref{24} has only solution $t=\dfrac{(\pm p-1)}{6}$ and $k=\dfrac{(p-1)}{2}$. Therefore, extracting the terms involving $q^{pn+19(p^{2}-1)/24}$, divding by $q^{19(p^{2}-1)/24}$ and replacing $q^{p}$ by $q$ from \eqref{22}, we obtain
\begin{equation}\label{25}
\sum_{n=0}^{\infty}P_v\left(6\cdot p^{2\alpha+1}n+\dfrac{19\cdot p^{2\alpha+2}+1}{4}\right)q^n \equiv 3\ell_{p}\ell_{6p}^3\pmod{6}.
\end{equation}Extracting the terms involving $q^{pn}$ and replacing $q^p$ by $q$ from \eqref{25}, we obtain
\begin{equation*}
\sum_{n=0}^{\infty}P_v\left(6\cdot p^{2\alpha+2}n+\dfrac{19\cdot p^{2\alpha+2}+1}{4}\right)q^n \equiv3\ell_1\ell_6^3\pmod{6},
\end{equation*}which is the $\alpha+1$ case of \eqref{e9}. Hence, by the method of induction we complete the proof of \eqref{e9}. Now, extracting the terms involving $q^{pn+j}$, $1\leq j\leq(p-1)$ from \eqref{25}, we arrive at \eqref{e10}.
\end{proof} 
\begin{theorem} We have 
\begin{equation*}
P_v(2n+1)= A_4(n)+A_4(n-2)+A_4(n-6)+A_4(n-12)+...+ A_4\Big(n-k(k+1)\Big)+...
\end{equation*}
\end{theorem}
\begin{proof}From \eqref{3}, we have
\begin{align}\sum_{n=0}^{\infty}P_v(2n+1)q^n
&=\dfrac{\ell_4}{\ell_1}\cdot\dfrac{\ell_4^2}{\ell_2} \notag \\
&=\left(\sum_{n=0}^{\infty}A_4(n)q^n\right)\psi(q^2)\notag \\
&=\left(\sum_{n=0}^{\infty}A_4(n)q^n\right)\left(\sum_{k=0}^{\infty}q^{k(k+1)}\right)\notag \\
\label{v2}
&=\left(\sum_{n=0}^{\infty}\sum_{k=0}^{\infty}A_4\Big(n-k(k+1)\Big)\right)q^n.
\end{align} On comparing the coefficients of $q^n$ in \eqref{v2}, we arrive at the desired result.
\end{proof}
\begin{theorem}We have
\begin{align*}
& P_v(6n+5)+\sum_{m=1}^{\infty}(-1)^m\Big(P_v(6n-9m^2-3m+5)+P_v(6n-9m^2+3m+5)\Big)\\
&\hspace{3.3cm}=3\sum_{t=0}^{\infty}\sum_{c=0}^{\lfloor \frac{n-3t^2-3t}{2} \rfloor}(-1)^t(2t+1)p_{2d}\left(n-3t^2-3t-2c\right)\bar{p}(c).
\end{align*}
\end{theorem}
\begin{proof}
From \eqref{2}, we note that
\begin{equation*}\left(\sum_{n=0}^{\infty}P_v(6n+5)q^n\right)\ell_1=3\left(\dfrac{\ell_4}{\ell_2^2}\right){\left(\dfrac{\ell_2}{\ell_1}\right)}^2\ell_6^3.\end{equation*}Employing \eqref{p1}, \eqref{p2}, \eqref{eq6} and \eqref{j},  we obtain 
\begin{align}
\hspace{-1cm}&\left(\sum_{n=0}^{\infty}P_v(6n+5)q^n\right)\left(1+\sum_{m=1}^\infty(-1)^mq^{m(3m+1)/2}+\sum_{m=1}^{\infty}(-1)^mq^{m(3m-1)/2}\right)\notag\\
&\hspace{3cm}=3\left(\sum_{k=0}^{\infty}\bar{p}(k)q^{2k}\right)\left(\sum_{l=0}^{\infty}p_{2d}(l)q^l\right)\ell_6^3\notag\\
&\hspace{3cm}=3\left(\sum_{l=0}^{\infty}\sum_{c=0}^{\lfloor \frac{l}{2}\rfloor}p_{2d}(l-2c)\bar{p}(c)q^l\right)\left(\sum_{n=0}^{\infty}(-1)^n(2n+1)q^{(3n^2+3n)}\right)\notag\\
&\hspace{3cm}=3\sum_{n=0}^{\infty}\sum_{l=0}^{\infty}\sum_{c=0}^{\lfloor\frac{l}{2}\rfloor}(-1)^n(2n+1)p_{2d}(l-2c)\bar{p}(c)q^{(l+3n^2+3n)},\notag
\end{align} which imples
$$\sum_{n=0}^{\infty}P_v(6n+5)q^n+\sum_{n=0}^{\infty}\sum_{m=1}^{\infty}(-1)^m\Big(P_v(6n-9m^2-3m+5)+P_v(6n-9m^2+3m+5)\Big)q^n $$\begin{equation}\label{r1}\hspace{4cm}=3\sum_{n=0}^{\infty}\sum_{t=0}^{\infty}\sum_{c=0}^{\lfloor\frac{n-3t^2-3t}{2}\rfloor}(-1)^t(2t+1)p_{2d}(n-3t^2-3t-2c)\bar{p}(c)q^n.\end{equation}Comparing the coefficients of $q^n$ in \eqref{r1}, we arrive at the desired result.
\end{proof}

\begin{remarks} Let $\mu(q)$ is as defined in \eqref{u} and  $\sum_{n=0}^{\infty}P_\mu(n)q^n=\mu(q)$. Then by \eqref{e5}, we have
	$$
	\sum_{n=0}^{\infty}P_\mu(n)q^n\equiv\dfrac{1}{\ell_1^3}=\sum _{n=0}^{\infty}p_3(n)q^n \pmod 4
	$$which implies $P_\mu(n)\equiv p_3(n) \pmod 4$.
\end{remarks}

\section{Results on $\sigma(q)$}
In this section, we set $\sum_{n=0}^{\infty}P_\sigma(n)q^n=\sigma(q)$, where $\sigma(q)$ is as defined in \eqref{s}.
\begin{theorem}We have
\begin{equation}\label{s1}
\sum_{n=0}^{\infty}{P_\sigma}(2n+1)q^n=\dfrac{\ell_2^2\ell_6^2}{\ell_1^2\ell_3}.
\end{equation}
\begin{proof}
Ramanujan listed linear relation connecting the sixth order mock theta functions with each other as~\cite{GM}:
\begin{equation}\label{s14}
\nu(q^2)-\sigma(-q)=q\dfrac{\ell_4^2\ell_{12}^2}{\ell_2^2\ell_6}, \end{equation}where $\nu(q)$ is the sixth order mock theta function defined by
\begin{equation*}
\nu(q)=\sum_{n=0}^{\infty}\dfrac{q^{n+1}(-q;q)_{2n+1}}{(q;q^2)_{n+1}}.
\end{equation*}Replacing $q$ by $-q$ in \eqref{s14}, we obatin
\begin{equation}\label{s2}
\sigma(q)=\nu(q^2)+q\dfrac{\ell_4^2\ell_{12}^2}{\ell_2^2\ell_6}.
\end{equation}Extracting the terms involving $q^{2n+1}$, dividing by $q$ and replacing $q^2$ by $q$ from \eqref{s2}, we arrive at the desired result.
\end{proof}
\begin{theorem}We have
$${P_\sigma}(2n+1)={P_\sigma}_{(d,e)}(n)-{P_\sigma}_{(d,o)}(n),$$
\end{theorem}where ${P_\sigma}_{(d,e)}(n)\left(resp. \quad {P_\sigma}_{(d,o)}(n)\right)$ denotes the number of partition of $n$ such that
\begin{itemize}
\item[$(i)$]parts congruent to $1,5\pmod{6}$ have two colour,
\item[$(ii)$]parts congruent to $3\pmod{6}$ have three colour,
\item[$(iii)$]parts congruent to $0\pmod{6}$ are distinct and have one colour,
\item[$(iv)$]number of parts congruent to $0\pmod{6}$ being even (resp. odd).
\end{itemize}
\end{theorem}
\begin{proof}
Proceeding in the same way as Theorem~\eqref{thmv}, we arrive at the desired result.
\end{proof}
\begin{theorem}
Let $p\ge5$ be any prime such that $\left(\dfrac{-2}{p}\right)=-1$. Then for integer $\alpha\ge0$, we have
\begin{equation}\label{s3}\hspace{-5cm}
\sum_{n=0}^{\infty}{P_\sigma}\left(2\cdot p^{2\alpha}n+\dfrac{11\cdot p^{2\alpha}+1}{12}\right)q^n\equiv\ell_2\psi(q^3)\pmod{2}.\end{equation}
\begin{equation}\label{s4}
{P_\sigma}\left(2\cdot p^{2\alpha+2}n+2\cdot p^{2\alpha+1}j+\dfrac{11\cdot p^{2\alpha+2}+1}{12}\right)\equiv0\pmod{2},\quad 1\leq j\leq(p-1).
\end{equation} 
\end{theorem}
\begin{proof}
Employing \eqref{lm1} and \eqref{eq5} in \eqref{s1}, we obtain
\begin{equation}\label{s5}
\sum_{n=0}^{\infty}{P_\sigma}(2n+1)q^n\equiv\ell_2\psi(q^3)\pmod{2}.
\end{equation}which is the $\alpha=0$ case of \eqref{s3}. Assume that \eqref{s3} is true for some $\alpha\geq0$. Now, substituting \eqref{eq7} and \eqref{eq8} in \eqref{s3}, we obtain
\begin{multline}\label{s6}
\sum_{n=0}^{\infty}{P_\sigma}\left(2\cdot p^{2\alpha}n+\dfrac{11\cdot p^{2\alpha}+1}{12}\right)\\ \equiv  3\Big[\sum_{\substack{t={-(p-1)/2} \\ t \ne {(\pm p-1)/6}}}^{(p-1)/2} (-1)^{t} q^{(3t^{2}+t)} f\left(-q^{3p^{2}+(6t+1)p}, -q^{3p^{2}-(6t+1)p}\right)+ (-1)^{(\pm p-1)/6} q^{(p^{2}-1)/12} f_{2p^2}\Big]\\ \hspace{-4.5cm}\times \Big[\sum_{m=0}^{(p-3)/2} q^{3(m^{2}+m)/2} f\left(q^{3{\left(p^{2}+(2m+1)p\right)/2}}, q^{3{\left(p^{2}-(2m+1)p\right)}/2}\right) \\+ q^{3(p^{2}-1)/8}\psi(q^{3p^2})\Big]\pmod{2}.
\end{multline} 
Consider the congruence
\begin{equation}\label{s7}
\left({3t^2+t}\right)+\left(\dfrac{3k^2+3k}{2}\right)\equiv 11\left(\dfrac{p^2-1}{24}\right)\pmod{p},
\end{equation}which is equivalent to \begin{equation}\label{s8}
(6t+3)^2+2(6k+1)^2\equiv0\pmod{p}.
\end{equation}For $\left(\dfrac{-2}{p}\right)=-1$, the congruence \eqref{s8} has only solution $t=\dfrac{(p-1)}{2}$ and $k=\dfrac{(\pm p-1)}{6}$. Therefore, extracting the terms involving $q^{pn+11(p^{2}-1)/24}$, divding by $q^{11(p^{2}-1)/24}$ and replacing $q^{p}$ by $q$ from \eqref{s6}, we obtain
\begin{equation}\label{s11}
\sum_{n=0}^{\infty}{P_\sigma}\left(2\cdot p^{2\alpha+1}n+\dfrac{11\cdot p^{2\alpha+2}+1}{12}\right)q^n\equiv\ell_{2p}\psi(q^{3p})\pmod{2}.
\end{equation}Extracting the terms involving $q^{pn}$ and replacing $q^p$ by $q$ from \eqref{s11}, we obtain
\begin{equation*} 
\sum_{n=0}^{\infty}{P_\sigma}\left(2\cdot p^{2\alpha+2}n+\dfrac{11\cdot p^{2\alpha+2}+1}{12}\right)q^n\equiv\ell_2\psi(q^3)\pmod{2},
\end{equation*}which is the $\alpha+1$ case of \eqref{s3}. Hence, by the method of induction we complete the proof of \eqref{s3}. Now, extracting the terms involving $q^{pn+j}$, $1\leq j\leq(p-1)$ from \eqref{s11}, we arrive at \eqref{s4}.
\end{proof}
\begin{theorem}We have
\begin{equation*}
{P_\sigma}(2n+1)=p_{2d}(n)+p_{2d}(n-3)+p_{2d}(n-9)+....+p_{2d}\left(n-\frac{3k^2+3k}{2}\right)+.....
\end{equation*}
\end{theorem}
\begin{proof}
From \eqref{s1}, we note that
$$\sum_{n=0}^{\infty}{P_\sigma}(2n+1)q^n
={\left(\dfrac{\ell_2}{\ell_1}\right)}^2\psi(q^3).$$ Employing \eqref{p2} and \eqref{eq5}, we obtain
\begin{align}\sum_{n=0}^{\infty}{P_\sigma}(2n+1)q^n
&=\left(\sum_{n=0}^{\infty}p_{2d}(n)q^n\right)\left(\sum_{k=0}^{\infty}q^{{3k(k+1)}/2}\right)\notag\\
\label{s13}
&=\sum_{n=0}^{\infty}\left(\sum_{k=0}^{\infty}p_{2d}\left(n-\frac{3k^2+3k}{2}\right)\right)q^n.
\end{align}On comparing the coefficients of $q^n$ in \eqref{s13}, we arrive at the desired result.
\end{proof}

\section{Results on $\beta(q)$}
Throughout the section,  we set $\sum_{n=0}^{\infty}P_\beta(n)q^n=\beta(q)$, where $\beta(q)$ is defined in \eqref{b}.
From \cite{ZS}, we also note that 
	\begin{equation}\label{b6}
	\sum_{n=0}^{\infty}P_\beta(3n+1)q^n=\dfrac{\ell_3^3}{\ell_1^2}. 
	\end{equation}

\begin{theorem}We have
\begin{equation}\label{b1}
\sum_{n=0}^{\infty}P_\beta(3n+2)q^n=2\dfrac{\ell_6^3}{\ell_1\ell_2}.
\end{equation}
\end{theorem}
\begin{proof}
Ramanujan listed linear relation connecting the sixth order mock theta functions with each other as~\cite{GM}:
\begin{equation}\label{b2}
\phi(q^3)+2q^{-1}\psi(q^3)+2\beta(q)=\dfrac{\ell_2\ell_3^5}{\ell_1^2\ell_6^3},
\end{equation}where $\phi(q)$ and $\psi(q)$ are the sixth order mock theta functions~\cite{GM} given by 
\begin{equation*}
\phi(q)=\sum_{n=0}^{\infty}\dfrac{(-1)^nq^{n^2}(q;q^2)_n}{(-q;q)_{2n}} \quad \text{and}\quad \psi(q)=\sum_{n=0}^{\infty}\dfrac{(-1)^nq^{(n+1)^2}(q;q^2)_n}{(-q;q)_{2n+1}}.
\end{equation*}Using \eqref{e2} in \eqref{b2}, we obtain
\begin{equation}\label{b10}
2\sum_{n=0}^{\infty}P_\beta(n)q^n=2\beta(q)=-\phi(q^3)-2q^{-1}\psi(q^3)+\dfrac{\ell_6\ell_9^6}{\ell_3^3\ell_{18}^3}+2q\dfrac{\ell_9^3}{\ell_3^2}+4q^2\dfrac{\ell_{18}^3}{\ell_3\ell_6}.
\end{equation}Extracting the terms involving $q^{3n+2}$, dividing by $q^2$ and replacing $q^3$ by $q$ from \eqref{b10}, we arrive at the desired result.
\end{proof}

\begin{theorem}
We have
$${P_\beta}(3n+2)={P_\beta}_{(d,e)}(n)-{P_\beta}_{(d,o)}(n),$$ where ${P_\beta}_{(d,e)}(n)\left(resp. \quad{P_\beta}_{(d,o)}(n)\right)$ denotes the number of partitions of $n$ such that
\begin{itemize}
\item[$(i)$]parts congruent to $1,3,5\pmod{6}$ have one colour,
\item[$(ii)$]parts congruent to $2,4\pmod{6}$ have two colour,
\item[$(iii)$]parts congruent to $0\pmod{6}$ are distinct and have one colour,
\item[$(iv)$]number of parts congruent to $0\pmod{6}$ being even (resp. odd).
\end{itemize}
\end{theorem}
\begin{proof}
Proceeding in the same way as Theorem~\eqref{thmv}, we arrive at the desired result.
\end{proof}
\begin{theorem}We have
\begin{equation}\label{b3}
{P_\beta}(9n+8)\equiv0\pmod{6}.
\end{equation}
\end{theorem}
\begin{proof}
Using \eqref{e4} in \eqref{b1}, we obtain
\begin{equation}\label{b4}
\sum_{n=0}^{\infty}{P_\beta}(3n+2)q^n=2\dfrac{\ell_9^9\ell_6}{\ell_3^6\ell_{18}^3}+2q\dfrac{\ell_9^6}{\ell_3^5}+6q^2\dfrac{\ell_9^3\ell_{18}^3}{\ell_3^4\ell_6}-4q^3\dfrac{\ell_{18}^6}{\ell_3^3\ell_6^2}+8q^4\dfrac{\ell_{18}^9}{\ell_3^2\ell_6^3\ell_9^3}.
\end{equation}Extracting the terms involving $q^{3n+2}$, dividing by $q^2$ and replacing $q^3$ by $q$ from \eqref{b4}, 
\begin{equation}\label{b5}
\sum_{n=0}^{\infty}{P_\beta}(9n+8)q^n=6\dfrac{\ell_3^3\ell_6^3}{\ell_1^4\ell_2}.
\end{equation}Now, the result follows immediately from \eqref{b5}.
\end{proof}
\begin{theorem}We have
\begin{align*}
&\sum_{k=0}^{\infty}{P_\beta}\left(3n-\frac{3k(k+1)}{2}+2\right)=2\bar{p}(n)-6\bar{p}(n-6)+10\bar{p}(n-18)-.....\\
&\hspace{8cm}+2(-1)^m(2m+1)\bar{p}\Big(n-3m(m+1)\Big)+....
\end{align*}
\end{theorem}
\begin{proof}
From \eqref{b1}, we note that
$$\left(\sum_{n=0}^{\infty}{P_\beta}(3n+2)q^n\right)\dfrac{\ell_2^2}{\ell_1}=2\left(\dfrac{\ell_2}{\ell_1^2}˘\right)\ell_6^3.$$
Employing \eqref{p1}, \eqref{eq5} and \eqref{j}, we obtain
$$\left(\sum_{n=0}^{\infty}{P_\beta}(3n+2)q^n\right)\left(\sum_{k=0}^{\infty}q^{k(k+1)/2}\right)=2\left(\sum_{n=0}^{\infty}\bar{p}(n)\right)\left(\sum_{m=0}^{\infty}(-1)^m(2m+1)q^{3m(m+1)}\right),$$ which implies
\begin{equation}\label{b7}
\sum_{n=0}^{\infty}\sum_{k=0}^{\infty}{P_\beta}\left(3n-\frac{3k(k+1)}{2}+2\right)q^n=2\sum_{n=0}^{\infty}\sum_{m=0}^{\infty}(-1)^m(2m+1)\bar{p}\Big(n-3m(m+1)\Big)q^n.
\end{equation}
On comparing the coefficients of $q^n$ in \eqref{b7}, we arrive at the desired result.
\end{proof}
\begin{theorem}We have
\begin{align*}
&{P_\beta}(9n+8)-3{P_\beta}(9n-10)+...+(-1)^m(2m+1){P_\beta}\Big(9(n-m^2-m)+8\Big)+...\\
&\hspace{5cm}=6\sum_{k=0}^{\infty}\sum_{l=0}^{\infty}(-1)^{l+k}(2k+1)(2l+1)\bar{p}\Big(n-\frac{3k^2+3k}{2}-3l^2-3l\Big).
\end{align*}
\end{theorem}
\begin{proof}From \eqref{b5}, we note that
$$\left(\sum_{n=0}^{\infty}{P_\beta}(9n+8)q^n\right)\ell_2^3=6{\left(\dfrac{\ell_2}{\ell_1^2}\right)}^2\ell_3^3\ell_6^3.$$ Employing \eqref{p1} and \eqref{j}, we obatin
\begin{align*}
&\left(\sum_{n=0}^{\infty}{P_\beta}(9n+8)q^n\right)\left(\sum_{m=0}^{\infty}(-1)^m(2m+1)q^{m(m+1)}\right)\\
&\hspace{2cm}=6\left(\sum_{n=0}^{\infty}\bar{p}_2(n)q^n\right)\left(\sum_{k=0}^{\infty}(-1)^k(2k+1)q^{(3k^2+3k)/2}\right)\left(\sum_{l=0}^{\infty}(-1)^l(2l+1)q^{(3l^2+3l)}\right),\\
\end{align*}
which implies
$$\hspace{-8cm}\sum_{n=0}^{\infty}\sum_{m=0}^{\infty}(-1)^m(2m+1)P_{\beta}\Big(9(n-m^2-m)+8\Big)q^n$$
\begin{equation}\label{b8}
\hspace{3cm}=6\sum_{n=0}^{\infty}\sum_{k=0}^{\infty}\sum_{l=0}^{\infty}(-1)^{l+k}(2k+1)(2l+1)\bar{p}\Big(n-\frac{3k^2+3k}{2}-3l^2-3l\Big).
\end{equation}
On comparing the coefficients of $q^n$ in \eqref{b8}, we arrive at the desired result.
\end{proof}
\begin{theorem}We have
\begin{align*}
&{P_\beta}(3n+1)+\sum_{m=1}^{\infty}(-1)^m{P_\beta}\Big(3n-3m(3m+1)+1\Big)+\sum_{m=1}^{\infty}(-1)^m{P_\beta}\Big(3n-3m(3m-1)+1\Big)\\
&\hspace{2cm}=\bar{p}(n)-3\bar{p}(n-3)+5\bar{p}(n-9)-....+(-1)^k(2k+1)\bar{p}\Big(n-\frac{3k(k+1)}{2}\Big)+....
\end{align*}
\end{theorem}
\begin{proof}from \eqref{b6}, we note that
$$\left(\sum_{n=0}^{\infty}{P_\beta}(3n+1)q^n\right)\ell_2=\left(\dfrac{\ell_2}{\ell_1^2}\right)\ell_3^3.$$
Employing \eqref{p1}, \eqref{eq6} and \eqref{j}, we obatin
\begin{align*}
&\left(\sum_{n=0}^{\infty}{P_\beta}(3n+1)q^n\right)\left(1+\sum_{m=1}^{\infty}(-1)^mq^{m(3m+1)}+\sum_{m=1}^{\infty}(-1)^mq^{m(3m-1)}\right)\\
&\hspace{3.8cm}=\left(\sum_{n=0}^{\infty}\bar{p}(n)q^n\right)\left(\sum_{k=0}^{\infty}(-1)^k(2k+1)q^{{3k(k+1)}/2}\right), 
\end{align*}which implies
$$\sum_{n=0}^{\infty}{P_\beta}(3n+1)q^n+\sum_{n=0}^{\infty}\sum_{m=1}^{\infty}(-1)^m\left({P_\beta}\Big(3n-3m(3m+1)+1\Big)+{P_\beta}\Big(3n-3m(3m-1)+1\Big)\right)q^n$$
\begin{equation}\label{b9}
\hspace{6.5cm}=\sum_{n=0}^{\infty}\sum_{k=0}^{\infty}(-1)^k(2k+1)\bar{p}\left(n-\frac{3k(k+1)}{2}\right)q^n.
\end{equation}
On comparing the coefficients of $q^n$ in \eqref{b9}, we arrive at the desired result.
\end{proof}

\section{Results on $\lambda(q)$}
Following three identities from \cite{RK} will be useful in this section: 
\begin{align}
\label{l2}
&\sum_{n=0}^{\infty}{P_\lambda}(2n)q^n=\dfrac{\ell_2^3\ell_3^2}{\ell_1^3\ell_6},\\
\label{l3}
&\sum_{n=0}^{\infty}{P_\lambda}(6n+2)q^n=3\dfrac{\ell_3^5}{\ell_6}{\left(\dfrac{\ell_2}{\ell_1^2}\right)}^3,\\
\label{l4}
&\sum_{n=0}^{\infty}{P_\lambda}(6n+4)q^n=\dfrac{\ell_2^2\ell_3^2\ell_6^2}{\ell_1^5},
\end{align}where 
$$\sum_{n=0}^{\infty}P_\lambda(n)q^n=\lambda(q)$$ and $\lambda(q)$ is defined as in \eqref{l}.

\begin{theorem}
We have
$${P_\lambda}(2n)={P_\lambda}_{(d,e)}(n)-{P_\lambda}_{(d,o)}(n),$$ where ${P_\lambda}_{(d,e)}(n)\left(resp.\quad{P_\lambda}_{(d,o)}(n)\right)$ denotes the number of partitions of $n$ such that
\begin{itemize}
\item[$(i)$]parts congruent to $1,5\pmod{6}$ have three colour,
\item[$(ii)$]parts congruent to $3\pmod{6}$ have one colour,
\item[$(iii)$]parts congruent to $0\pmod{6}$ are distinct and have one colours,
\item[$(iv)$]number of parts congruent to $0\pmod{6}$ being even (resp. odd).
\end{itemize}
\end{theorem}
\begin{proof}
Proceeding in the same way as Theorem~\ref{thmv}, we arrive at the desired result.
\end{proof}
\begin{theorem}We have
\begin{equation*}
{P_\lambda}(2n)=p_{3d}(n)-2p_{3d}(n-3)+2p_{3d}(n-12)-.....+2(-1)^kp_{3d}(n-3k^2)+.....
\end{equation*}
\begin{proof}
From \eqref{l2}, we note that
$$\sum_{n=0}^{\infty}{P_\lambda}(2n)q^n
={\left(\dfrac{\ell_2}{\ell_1}\right)}^3\phi(-q^3).$$
Employing \eqref{p2} and \eqref{eq4}, we obtain
\begin{align}\sum_{n=0}^{\infty}{P_\lambda}(2n)q^n
&=\left(\sum_{n=0}^{\infty}p_{3d}(n)q^n\right)\left(\sum_{k=-\infty}^{\infty}(-1)^kq^{3k^2}\right)\notag\\
&=\left(\sum_{n=0}^{\infty}p_{3d}(n)q^n\right)\left(1+2\sum_{k=1}^{\infty}(-1)^kq^{3k^2}\right)\notag\\
\label{l7}
&=\sum_{n=0}^{\infty}p_{3d}(n)q^n+2\sum_{n=0}^{\infty}\sum_{k=1}^{\infty}(-1)^kp_{3d}(n-3k^2)q^n.
\end{align}On comparing the coefficients of $q^n$ in \eqref{l7}, we arrive at the desired result.
\end{proof}
\end{theorem}
\begin{theorem}We have
\begin{align*}
&{P_\lambda}(6n+2)=3\bar{p}_{3}(n)-9\bar{p}_{3}(n-3)+...+3(-1)^l(2l+1)\bar{p}_{3}\left(n-\dfrac{3l(l+1)}{2}\right)+....\\
&\hspace{4.5cm}..+6\sum_{l=0}^{\infty}\sum_{k=1}^{\infty}(-1)^{l+k}(2l+1)\bar{p}_{3}\left(n-3k^2-\dfrac{3l(l+1)}{2}\right).
\end{align*}
\end{theorem}
\begin{proof}
From \eqref{l3}, we note that
$$\sum_{n=0}^{\infty}{P_\lambda}(6n+2)q^n
=3{\left(\dfrac{\ell_2}{\ell_1^2}\right)}^3\phi(-q^3)\ell_3^3.$$ Employing \eqref{p1}, \eqref{eq4} and \eqref{j}, we obtain
\begin{align*}\sum_{n=0}^{\infty}{P_\lambda}(6n+2)q^n
&=3{\left(\dfrac{\ell_2}{\ell_1^2}\right)}^3\left(1+2\sum_{k=1}^{\infty}(-1)^kq^{3k^2}\right)\left(\sum_{l=0}^{\infty}(-1)^l(2l+1)q^{{3l(l+1)}/2}\right)\\
&=3\left(\sum_{n=0}^{\infty}\bar{p}_{3}(n)q^n\right)\left(\sum_{l=0}^{\infty}(-1)^l(2l+1)q^{{3l(l+1)}/2}\right)\\
&\hspace{1cm}+6\left(\sum_{n=0}^{\infty}\bar{p}_{3}(n)q^n\right)\left(\sum_{l=0}^{\infty}\sum_{k=1}^{\infty}(-1)^{l+k}(2l+1)q^{3k^2+{3l(l+1)}/2}\right)\\
&=3\sum_{n=0}^{\infty}\sum_{l=0}^{\infty}(-1)^l(2l+1)\bar{p}_{3}\left(n-\dfrac{3l(l+1)}{2}\right)q^n\end{align*}
\begin{equation}\label{l5}
\hspace{4.9cm}+6\sum_{n=0}^{\infty}\sum_{l=0}^{\infty}\sum_{k=1}^{\infty}(-1)^{l+k}(2l+1)\bar{p}_{3}\left(n-3k^2-\dfrac{3l(l+1)}{2}\right)q^n.
\end{equation}
On comparing the coefficients of $q^n$ in \eqref{l5}, we arrive at the desired result.
\end{proof}
\begin{theorem}We have
\begin{align*}
&{P_\lambda}(6n+4)-3{P_\lambda}(6n-2)+...+(-1)^m(2m+1){P_\lambda}\Big(6n-3m(m+1)+4\Big)+...\\
&\hspace{-.2cm}=6\sum_{l=0}^{\infty}(-1)^l(2l+1)p_{2d}\Big(n-3l(l+1)\Big)+12\sum_{l=0}^{\infty}\sum_{k=1}^{\infty}(-1)^{l+k}(2l+1)p_{2d}\Big(n-3k^2-3l(l+1)\Big).
\end{align*}
\end{theorem}
\begin{proof}From \eqref{l4}, we note that
\begin{equation}\label{l8}
\left(\sum_{n=0}^{\infty}{P_\lambda}(6n+4)q^n\right)\ell_1^3=6{\left(\dfrac{\ell_2}{\ell_1}\right)}^2\phi(-q^3)\ell_6^3.
\end{equation} Employing \eqref{eq4} and \eqref{j} in \eqref{l8}, we obtain
\begin{equation}\label{l9}
\left(\sum_{n=0}^{\infty}{P_\lambda}(6n+4)q^n\right)\ell_1^3=6{\left(\dfrac{\ell_2}{\ell_1}\right)}^2\left(1+2\sum_{k=1}^{\infty}(-1)^kq^{3k^2}\right)\left(\sum_{l=0}^{\infty}(-1)^l(2l+1)q^{3l(l+1)}\right).
\end{equation} Again, emlpoying \eqref{p2} and \eqref{j} in \eqref{l9}, we obtain
\begin{align*}
&\left(\sum_{n=0}^{\infty}{P_\lambda}(6n+4)q^n\right)\left(\sum_{m=0}^{\infty}(-1)^m(2m+1)q^{{m(m+1)}/2}\right)\\
&\hspace{1cm}=6\left(\sum_{n=0}^{\infty}p_{2d}(n)q^n\right)\left(\sum_{l=0}^{\infty}(-1)^l(2l+1)q^{3l(l+1)}+2\sum_{l=0}^{\infty}\sum_{k=1}^{\infty}(-1)^{l+k}(2l+1)q^{3k^2+3l(l+1)}\right).\end{align*}
\newpage
which implies
\begin{align}\label{l6}
&\hspace{-1cm}\sum_{n=0}^{\infty}\sum_{m=0}^{\infty}(-1)^m(2m+1){P_\lambda}\Big(6n-3m(m+1)+4\Big)q^n\notag\\
&\hspace{1.5cm}=6\sum_{n=0}^{\infty}\sum_{l=0}^{\infty}(-1)^l(2l+1)p_{2d}\Big(n-3l(l+1)\Big)q^n\notag\\
&\hspace{2cm}+12\sum_{n=0}^{\infty}\sum_{l=0}^{\infty}\sum_{k=1}^{\infty}(-1)^{l+k}(2l+1)p_{2d}\left(n-3k^2-3l(l+1)\right)q^n.
\end{align}
On comparing the coefficients of $q^n$ in \eqref{l6}, we arrive at the desired result.
\end{proof}

\section*{Acknowledgements} The first author acknowledges the financial support received from University Grants Commission through National Fellowship for Scheduled Caste  Students under grant Ref. no.: 211610029643.

\section*{\bf Declarations}
\noindent{\bf Conflict of Interest.} The authors declare that there is no conflict of interest regarding the publication of
this article.

\noindent{\bf Human and animal rights.} The authors declare that there is no research involving human participants or
animals in the contained of this paper.	

\noindent{\bf Data availability statements.} Data sharing not applicable to this article as no datasets were generated or analysed during the current study.	

\end{document}